\documentclass{amsproc}

\usepackage{amsmath}
\usepackage{amsthm}
\usepackage{amssymb}
\usepackage[fleqn,tbtags]{mathtools}
\usepackage{fixmath}
\usepackage{graphicx}
\usepackage{color}
\usepackage[all]{xy}
\usepackage{bm}
\usepackage{mathrsfs}

\usepackage[pdfstartpage=1, pdfstartview=XYZ,]{hyperref}

\newtheorem{thm}{Theorem}[section]
\newtheorem{lem}[thm]{Lemma}
\newtheorem{prop}[thm]{Proposition}
\newtheorem{cor}[thm]{Corollary}

\theoremstyle{definition}
\newtheorem{definition}[thm]{Definition}

\theoremstyle{remark}
\newtheorem{rem}[thm]{Remark}
\newtheorem{ass}[thm]{Assumption}
\newtheorem{summary}[thm]{Summary}

\numberwithin{equation}{section}

\usepackage{enumitem}

\newlength\mylen
\usepackage{caption}
\usepackage{subcaption}

\newcommand\scalemath[2]{\scalebox{#1}{\mbox{\ensuremath{\displaystyle #2}}}}

\begin{document}

\title[Point Counting on Genus 3 Curves with Automorphism Group $\mathbb{Z} / 2 \mathbb{Z}$]{Point Counting on Non-Hyperelliptic Genus 3 Curves with Automorphism Group $\mathbb{Z} / 2 \mathbb{Z}$ using Monsky-Washnitzer Cohomology}

\author{Yih-Dar SHIEH}

\thanks{This work was supported by a project of the Agence Nationale de la Recherche, reference ANR-12-BS01-0010-01.}
\address{Aix Marseille Universit\'e, CNRS, Centrale Marseille, I2M, UMR 7373, 13453 Marseille, France}
\curraddr{\newline
Institut de Math\'ematiques de Marseille (I2M)\newline
163 avenue de Luminy, Case 907\newline
13288 Marseille, cedex 9\newline France}
\email{chiapas@gmail.com}

\begin{abstract}
 We describe an algorithm to compute the zeta function of any non-hyperelliptic genus 3 plane curve $C$ over a finite field with
automorphism group $G = \mathbb{Z} / 2 \mathbb{Z}$. This algorithm computes in the Monsky-Washnitzer cohomology of~the curve. Using the
relation between the Monsky-Washnitzer cohomology of $C$ and its quotient $E := C/G$, the computation splits into 2
parts: one in a subspace of the Monsky-Washnitzer cohomology and a second which reduces to the point counting on an elliptic curve $E$. The former corresponds to the dimension $2$ abelian surface $\mathrm{ker}(\mathrm{Jac}(C) \rightarrow E)$, on which we can compute with lower precision
and with matrices of smaller dimension. Hence we obtain a faster algorithm than working directly on the curve $C$.
\end{abstract}

\maketitle

\section{Introduction}

Henn gave the table of the possible non-trivial groups which appear as automorphism groups of a non-hyperelliptic
genus 3 curves, which can be found in Vermeulen's thesis \cite{Vermeulen}. The dimension of the set of moduli points
of non-hyperelliptic genus 3 curves whose automorphism group contain $\mathbb{Z}/2\mathbb{Z}$ is 4 inside the moduli of genus 3 curves $\mathcal{M}_3$ of dimension 6. We thus obtain an algorithm to compute the zeta
function of a large family of genus 3 curves.

In \cite{Kedlaya2}, Kedlaya used Monsky-Washnitzer cohomology to compute the zeta functions of hyperelliptic
curves over finite fields. This method could be applied to general varieties, and there are already generalizations to
superelliptic curves, $C_{a,b}$ curves and non-degenerate curves, see \cite{Castryck}, \cite{D-V} and \cite{Gaudry}.
This work also uses Monsky-Washnitzer cohomology but focuses on a smaller dimensional space associated to an abelian surface in the jacobian of $C$.

The remainder of this paper is organized as follows: Section \ref{MW} recalls the definition of Monsky-Washnitzer cohomology and results in this theory. In Section~\ref{Cohomology}, we compute a basis of the cohomology and describe a way to do the reduction of a differential form. Reduction means to write a differential form as a linear combination of the basis. In particular, we give an upper bound of the denominator after a differential form is reduced. This bound makes the algorithm
practical since it establishes a finite precision bound for the computation. Section \ref{Frobenius} describes a way to compute a lift of Frobenius and Section \ref{quotient} explains why the computation splits into $2$ eigenspaces. Finally, Section \ref{Algorithm} gives the algorithm and an analysis of its complexity.

\section{Monsky-Washnitzer Cohomology} \label{MW}

In this section, we recall the definition of Monsky-Washnitzer cohomology which is introduced by Monsky
and Washnitzer in \cite{M-W}\cite{Monsky 1}\cite{Monsky 2}.

Monsky-Washnitzer  cohomology is a $p$-adic cohomology theory defined for smooth affine varieties over
finite fields. Let $X$ be a smooth affine variety defined over a finite field $k:= \mathbb{F}_q$ of
characteristic $p$ with coordinate ring $\overline{A}$ which is a finitely generated $k$-algebra. In
\cite{Elkik}, Elkik showed that there exists a finitely generated smooth $\mathbb{Z}_q$-algebra $A$ such
that $A/pA \cong \overline{A}$, here $\mathbb{Z}_q$ is the valuation ring of $\mathbb{Q}_q$, the degree
$n:= \mathrm{log}_p q$ unramified extension of $\mathbb{Q}_p$.

In general, $A$ does not admit a lift of the Frobenius endomorphism $\overline{F}$ on
$\overline{A}$ , but its $p$-adic completion $A^{\infty}$ does. However, the dimension of the de Rham
cohomology of $A^ {\infty}$ may be too big. For example, if $A = \mathbb{Z}_p [x]$, then $
\sum_{n=0}^{\infty} p^n x^{p^n-1} dx$ is not an exact differential form since $\sum_{n=0}^{\infty}
x^{p^n}$ is not in $A^{\infty}$, but each term of this sum is exact. The problem is that this
differential form does not converge fast enough for its integral to converge as well.

Monsky and Washnitzer work with a subalgebra $A^{\dag}$ of $A^{\infty}$ consisting of series which converge fast
enough to solve the above problem. For
\[
A = \mathbb{Z}_q [x_1, x_2, \cdots, x_d]/(f_1, f_2, \cdots, f_r),
\]
the weak completion or dagger ring of $A$ is
\[
A^{\dag} := \mathbb{Z}_q \langle x_1, x_2, \cdots, x_d \rangle / (f_1, f_2, \cdots, f_r)
\]
where $\mathbb{Z}_q \langle x_1, x_2, \cdots, x_d \rangle$ is the subring of $A^{\infty}$ which consists
of overconvergent power series
\[
\left\{ \sum_{\alpha} a_{\alpha} x^{\alpha} \in \mathbb{Z}_q [[ x_1, x_2, \cdots, x_d ]] \ \Big| \  \liminf_{|\alpha| \rightarrow \infty} \frac{v_{p}(a_\alpha)} {|\alpha|} > 0 \right\}
\]
with $\alpha := (\alpha_1, \cdots, \alpha_d)$, $x^{\alpha} := x_1^{\alpha_1} \cdots x_d^{\alpha_d}$,
$|\alpha| = \sum_{i=1}^{d} \alpha_i$ and $v_p$ is the usual $p$-adic valuation on $\mathbb{Z}_q$.

\begin{definition}

The Monsky-Washnitizer cohomology of $X/\mathbb{F}_q$ is the de Rham cohomology of $A^{\dag} \otimes_{
\mathbb{Z}_q}  \mathbb{Q}_q$. More precisely, let $D^{0}(A^{\dag}) := A^{\dag}$, $D^{1}(A^{\dag})$ be
its universal module of differentials
\[
D^{1}(A^{\dag}) := \left( A^{\dag} \ dx_1 + \cdots + A^{\dag} \ dx_d \right) / \left( \sum_{j=1}^{r} A^{\dag} \left( \frac{\partial f_i}{\partial x_1} \, dx_1 + \cdots + \frac{\partial f_i}{\partial x_d} \, dx_d \right) \right)
\]
and $D^{i}(A^{\dag})$ be the $i$-th exterior product of $D^{1}(A^{\dag})$. Let $H^i(\overline{A} ,
\mathbb{Z}_q)$ be the $i$-th cohomology group of the complex
\[
0    \xrightarrow{} D^{0}(A^{\dag}) \xrightarrow{d_0}  D^{1}(A^{\dag}) \xrightarrow{d_1} D^{2}(A^{\dag}) \xrightarrow{d_2} D^{3}(A^{\dag}) \xrightarrow{d_3} \cdots
\]
where $d_i$ is the usual differentiation. Then the $i$-th Monsky-Washnitzer cohomology group of $X$ (or
of $\overline{A}$) is $  H^i(\overline{A} , \mathbb{Z}_q) \otimes_{
\mathbb{Z}_q} \mathbb{Q}_q $, which is denoted by $H^i_{\scriptscriptstyle{MW}} (X/\mathbb{F}_q)$ (or $H^i(\overline{A} , \mathbb{Q}_q)$).

\end{definition}

The Monsky-Washnitzer cohomology has the following properties, see \cite{van der Put}.

\begin{thm}
For a smooth finitely generated $\mathbb{F}_q$-algebra $\overline{A}$, we have
\begin{enumerate}[topsep=1.5mm] \setlength{\itemsep}{-0.0mm}
    \item The map $\overline{A} \mapsto H^i(\overline{A} , \mathbb{Q}_q) $ is well defined and
        functorial.
    \item There exists a $\mathbb{Z}_q$-algebra homomorphism $F_q$ : $A^{\dag} \rightarrow A^{\dag}$
        which lifts the Frobenius endomorphism of $\overline{A}$. Furthermore, any two lifts induce
        homotopic maps on the complex $D^{i}(A^{\dag})$. Hence they induce the same map $F_{q,*} :
        H^i(\overline{A} , \mathbb{Q}_q) \rightarrow H^i(\overline{A} , \mathbb{Q}_q)$ on the
        Monsky-Washnitzer cohomology.
\end{enumerate}
\end{thm}
The following Lefschetz fixed point formula allows us to compute the zeta function of $X =
\mathrm{Spec}(\overline{A})$ using Monsky-Washnitzer cohomology.

\begin{thm}[Lefschetz fixed point formula] \label{thm3}
Let $X/\mathbb{F}_q$ be a smooth affine variety of dimension $d$. Then we have
\[ X(\mathbb{F}_q) =
\sum_{i=0}^{d} (-1)^i \mathrm{Tr} \left( q^d F_{q,*}^{-1} \Big| H_{\scriptscriptstyle{MW}}^i(X/\mathbb{F}_q) \right) .
\]
\end{thm}

\section{Cohomology of Non-Hyperelliptic Genus 3 Plane Curves with Automorphism Group $\mathbb{Z} / 2 \mathbb{Z}$} \label{Cohomology}

In this article, we consider non-hyperelliptic smooth projective plane curves $C$ of genus 3 whose automorphism group
contains $\mathbb{Z} / 2 \mathbb{Z}$ over a finite field $\mathbb{F}_q$ of characteristic $p \neq 2$. Such curves can be written (up to isomorphism) as
\[
C : \overline{F} := Y^4+ \overline{G}(X,Z)X^2+ \overline{H}(X,Z) = 0,
\]
with $\overline{G}(X,Z)$ and $\overline{H}(X,Z) \in \mathbb{F}_q [X,Z]$ which are homogeneous of degree
2 and 4 respectly. We assume that $C$ is smooth. Since the Monsky-Washnitzer cohomology is defined for
smooth affine varieties, we also consider the affine part of $C$
\[
C_{\mathrm{aff}} : \overline{f} := y^4+ \overline{g}(x)y^2+ \overline{h}(x) = 0 ,
\]
where $\overline{g}(x) = \overline{G}(x,1)$ and $\overline{h}(x) = \overline{H}(x,1)$ are the
dehomogenizations of $\overline{G}(X,Z)$ and $\overline{H}(X,Z)$ with respect to $Z$. In this section, we compute the
Monsky-Washnitzer cohomology $H_{\scriptscriptstyle{MW}}^i(C_{\mathrm{aff}} / \mathbb{F}_q)$ of
$C_{\mathrm{aff}}$ and relate the zeta function of $C/\mathbb{F}_q$ to the characteristic polynomial of
the Frobenius action $F_{q,*}$ on $H_{\scriptscriptstyle{MW}}^i(C_{\mathrm{aff}}/\mathbb{F}_q)$.

Choose arbitrary lifts $G(X,Z)$ and $H(X,Z) \in \mathbb{Z}_q [X,Z]$  of $\overline{G}(X,Z)$ and
$\overline{H}(X,Z)$ such that $ \mathrm{deg}_{X} G = \mathrm{deg}_{X} \overline{G}$ and $
\mathrm{deg}_{X} H = \mathrm{deg}_{X} \overline{H}$. Let $g(x) := G(x,1)$ and $h(x) := H(X,1)$ be the
dehomogenizations. Consider the following two curves
\[
\tilde{C} : F := Y^4+ G(X,Z)Y^2+ H(X,Z) = 0,
\]
and
\begin{equation}
\tilde{C}_{\mathrm{aff}} : f := y^4+ g(x)y^2+ h(x) = 0  \label{eq1}
\end{equation}
Since the reduction of $F$ modulo the maximal ideal $(p)$ of $\mathbb{Z}_q$ is $\overline{F}$ which
defines a smooth projective curve $C$, the generic fiber $\tilde{C}_{\xi} := \tilde{C}
\times_{\mathbb{Z}_q} \mathbb{Q}_q$ of $\tilde{C}$ is also smooth. Using the three facts that the
reduction of $f$ modulo the maximal ideal $(p)$ of $\mathbb{Z}_q$ equals to $\overline{f}$ which is not
zero in $\mathbb{F}_q [x,y]$, that $\overline{A} := \mathbb{F}_q [x,y]/(\overline{f}(x,y))$ is an integral
domain and that $p$ is a prime element in the unique factorization domain $\mathbb{Z}_q [x,y]$, one sees that
$A$ is an integral domain and hence it is flat over $\mathbb{Z}_q$. This shows that $A$ is a finitely
generated smooth $\mathbb{Z}_q$-algebra, so we can work with $A$ to apply the theory of
Monsky-Washnitzer cohomology. The above arguments also show that the generic fiber $\tilde{C}_{\xi}$ of
$\tilde{C}$ is a geometrically integral smooth projective curve over $\mathbb{Q}_q$.

Although we can compute the Monsky-Washnitzer cohomology of the affine curve $ C_{\mathrm{aff}} /
\mathbb{F}_q $ by explicit reduction algorithms and the control of denominators, we use the following
theorem instead, see \cite{Kedlaya}, and compute the algebraic de Rham cohomology
$H^{i}_{\scriptscriptstyle{dR}} (\tilde{C}_{\xi, \mathrm{aff}}/ \mathbb{Q}_q )$ of the curve
$\tilde{C}_{\xi, \mathrm{aff}}/\mathbb{Q}_q$, the affine part of the generical fiber $\tilde{C}_{\xi}$.
Note that we are concerned with curves, hence the divisors are always normal crossings.

\begin{thm} \label{thm1}

Let $Y/\mathbb{Z}_q$ be a smooth proper scheme, $Z$ be a relative normal crossings divisor and $X := Y
\backslash Z$ is affine. Then there is a canonical isomorphism
\[
H^{i}_{\scriptscriptstyle{dR}} ( X_{\xi} / \mathbb{Q}_q ) \rightarrow H^{i}_{\scriptscriptstyle{MW}} ( X_{p} / \mathbb{F}_q ),
\]
where $X_{\xi}$ is the generic fiber and $X_{p}$ is the special fiber of $X/\mathbb{Z}_q$, namely, the
fibers of $X$ at the closed point $(p)$ of $\mathrm{Spec}(\mathbb{Z}_q)$.

\end{thm}

\subsection{Geometry} \label{Geometry}
Before we compute the algebraic de Rham cohomology of the affine curve $\tilde{C}_{\xi, \mathrm{aff}}$, we need to study its
geometry. The coordinate ring of $ \tilde{C}_{\xi, \mathrm{aff}} $ is $A := \mathbb{Q}_q
[x,y]/(f(x,y))$, here $ f(x,y) = y^4+ g(x)y^2+ h(x) \in \mathbb{Z}_q[x,y]$ with $\mathrm{deg}(g) \leq 2$
and $\mathrm{deg}(h) \leq 4$. We write $g(x)$ and $h(x)$ as following
\begin{equation}
  \begin{split}
     g(x) &= a_2 x^2 + a_1 x + a_0, \ a_i \in \mathbb{Z}_q \\
     h(x) &= b_4 x^4 + b_3 x^3 + b_2 x^2 + b_1 x + b_0, \ b_i \in \mathbb{Z}_q
  \end{split} \label{eq2}
\end{equation}
\\
There are four cases to consider:
\begin{enumerate}[label=\textbf{Case~\arabic*.},
ref={Case~\arabic*},labelwidth=\dimexpr-\mylen-\labelsep\relax,leftmargin=0pt,align=right]
\item $ b_4 = 0 $ and $ a_2 = 0 $ \label{case5} \\
There is only one point at infinity which is $ P_{\infty} := (1:0:0) $. Using the fact that $
\tilde{C}_{\xi} $ is smooth at $P_{\infty} $, one shows that $ b_3 \neq 0 $. The dimension of the
first algebraic de Rham cohomology of $ \tilde{C}_{\xi, \mathrm{aff}} / \mathbb{Q}_q $ is $ 2g + N_{
\infty } - 1 = 2 \cdot 3 + 1-1 = 6 $, here $ g $ is the genus of $ \tilde{C}_{\xi} $, which equals
to the genus of $C$, and $ N_{\infty} $ is the number of points at infinity. We have
$\upsilon_{P_{\infty}}(x) = -4$ and $\upsilon_{P_{\infty}}(y) = -3$. The local parameter at
$P_{\infty}$ is $t := b_3 x^2/y^3$. The expansions of $x$ and $y$ as Laurent series of the local
parameter $t$ are $x=-b_3/t^4+\cdots$ and $y = b_3/t^3+\cdots$.

\item $ b_4 = 0 $ and $a_2 \neq 0$ \label{case6} \\
There are 3 points at infinity: $ P_{\infty} :=(1:0:0)$ and $ P_{\infty, \pm} := (1 :\pm \alpha:0) $
with $\alpha^2 = -a_2 $. Using the fact that $ \tilde{C}_{\xi} $ is smooth at $P_{\infty} $, one
shows that $ b_3 \neq 0 $. (The condition $ a_2 \neq 0 $ implies the smoothness at $ P_{\infty, +} $
and $ P_{\infty, -}$.) We have $\mathrm{dim}_{\mathbb{Q}_q} H^{1}_{\scriptscriptstyle{dR}} ( \tilde{C}_{\xi, \mathrm{aff}} /
\mathbb{Q}_q ) = 2 \cdot 3 + 3 - 1 = 8 $, $\upsilon_{P_{\infty, \pm}}(x) = \upsilon_{P_{\infty,
\pm}}(y) = -1$, $\upsilon_{P_{\infty}}(x) = -2$ and $\upsilon_{P_{\infty}}(y) = -1$. The local
parameters at $P_{\infty}$ and $P_{\infty, \pm}$ are $t := 1/y$ and $t_{\pm} := 1/x$. The expansions
of $x$ and $y$ at $P_{\infty}$ and $P_{\infty, \pm}$ as Laurent series of the local parameters are $x = \beta/t^2 + \gamma + \delta t^2 + \cdots$ with $\beta = -a_2/b_3$, $y
= 1/t$, and $x = 1/t_{\pm}$,
$y= \pm \alpha/t + \cdots$.

\item $ b_4 \neq 0 $ and $ a_2^2 - 4 b_4 = 0$ \label{case3} \\
There are 2 points at infinity: $ P_{\infty, \pm} := (1: \pm \alpha:0) $ with $\alpha = (-a_2/2
 )^{1/2}$. Using the fact that $ \tilde{C}_{\xi} $ is smooth at $P_{\infty, +} $, one shows that $
 a_1 a_2 - 2 b_3 \neq 0$.  We have $\mathrm{dim}_{\mathbb{Q}_q} H^{1}_{\scriptscriptstyle{dR}} ( \tilde{C}_{\xi, \mathrm{aff}} /
 \mathbb{Q}_q )   = 2 \cdot 3 + 2 - 1 = 7 $ and $\upsilon_{P_{\infty, \pm}}(x) = \upsilon_{P_{\infty,
 \pm}}(y) = -2$. The local parameters are $ t_{ \pm } : = y/x \mp \alpha$. The expansions of $x$ and
 $y$ as Laurent series of the local parameters are $x = \beta/t_{\pm}^2+\gamma/t_{\pm} + \cdots$ and
 $y = \pm \alpha \beta / t_{\pm}^2 + (\beta \pm \alpha \gamma)/t_{\pm} + \cdots$ with $\beta =
 -(a_1a_2-2b_3)/4a_2$ and $\gamma = \alpha(a_1 a_2 + 2b_3)/2 a_2^2$.

\item $ b_4 \neq 0 $ and $ a_2^2 - 4 b_4 \neq 0 $ \label{case7} \\
 There are 4 points at infinity which are $ P_{ \infty,i, \pm }  := (1: \pm \alpha_i:0 )$ with $ \pm
 \alpha_1$ and $ \pm \alpha_2$ are the four roots of $y^4+a_2y^2+b_4 = 0$. We have $
 \mathrm{dim}_{\mathbb{Q}_q} H^{1}_{\scriptscriptstyle{dR}} ( \tilde{C}_{\xi, \mathrm{aff}} / \mathbb{Q}_q ) = 2 \cdot 3
 + 4 - 1 = 9 $ and $\upsilon_{P_{\infty, i, \pm}}(x) = \upsilon_{P_{\infty, i, \pm}}(y) = -1$. The
 local parameters are $t := 1/x$. The expansions of $x$ and $y$ as Laurent series of $t$ are $x = 1/t$
 and $y = \pm \alpha_i/t+\cdots$.
\end{enumerate}

In order to analyze the control of denominators later, we need to impose further assumptions on the choice of the lift
$F(X,Y,Z)$.

\begin{ass}
The coefficients $a_i$ and $b_j$ of $g(x)$ and $h(x)$ in {\rm(\ref{eq2})} are either $0$ or units in
$\mathbb{Z}_q$. Furthermore, $a_2^2 - 4 b_4$ is either 0 or a unit in $\mathbb{Z}_q$.
\end{ass}

\begin{rem} \label{rem1}
A lift which satisfies the above assumptions could be constructed by using Teichm\"{u}ller lift.
The choice of such a lift is to preserve the geometric structure.
Under these assumptions, we introduce some facts which will be used later. The expansions of $x$ and $y$ as Laurent series of the local parameters have integral
coefficients.\footnote{Use Hensel's lemma. For \ref{case5}, one needs a 2-variable version of Hensel's
lemma.} This means that $x$ and $y$ are in $\mathcal{O}((t))$, here $\mathcal{O}$ is the integral
closure of $\mathbb{Z}_q$ in a finite extension $\mathbb{Q}_q(\alpha)$  {\rm (}$\mathbb{Q}_q( \alpha_1,
\alpha_2)$ in {\rm \ref{case7}}{\rm )} of $\mathbb{Q}_q$ and $\alpha$ is the $Y$-coordinate of the
points at infinity discussed above. Furthermore, the coefficients of the lowest
terms\,\footnote{\hspace{0.9pt}In particular, $\alpha, \beta \in \mathcal{O}^{\ast}$ if we use the
notations in the above classification.} in these Laurent series are units in $\mathcal{O}$. {\rm (}In {
\rm \ref{case3}}, one shows that $a_1a_2-2b_3$ is a unit in $\mathbb{Z}_q$ by using $C$ is smooth.{\rm
)} $\tilde{C}_{\xi}$ and $\overline{C}$ have the same geometry. In {\rm \ref{case7}}, $\sqrt{a_2^2-4
b_4}$ is a unit in $\mathcal{O}$.
\end{rem}

\subsection{The Reduction Algorithm and Algebraic de Rham Cohomology} \label{Reduction}

In this subsection, we present the reduction algorithm and use it to compute $ H^{1}_{\scriptscriptstyle{dR}} (
\tilde{C}_{\xi, \mathrm{aff}} / \mathbb{Q}_q ) $. First of all, since $ (i+1) x^j y^i dy = d (x^j y^{i+1}) - j x^{j-1} y^{i+1} dx$ for all $ i \geq 0$, the universal module of differential $
\mathrm{\Omega}^{1} $ of $\tilde{C}_{\xi, \mathrm{aff}} / \mathbb{Q}_q$ is generated by $ \{ x^j y^i dx
\ | \ i \geq 0, j \geq 0 \}  $. From the defining equation (\ref{eq1}), $\mathrm{\Omega}^1$ is generated
by $ \{ x^j y^i dx \ | \ 1 \leq i \leq 3, j \geq 0 \}  $. Since
\[
0 = df = d \left( y^4+g(x)y^2+h(x) \right) = \left( g'(x)y^2+h'(x) \right) dx + \left( 4y^3+2g(x)y \right) dy,
\]
we have
\[
x^k \left( g'(x)y^2+h'(x)  \right) y^l dx + x^k \left( 4y^3+2g(x)y \right) y^l dy = 0.
\]
Combine with the following equation
\begin{align*}
&\mathrel{\phantom{=}}  d \left( x^k \left( \frac{4}{l+4} y^{l+4}+\frac{2}{l+2} g(x) y^{l+2} \right) \right)  -  x^k \left( 4y^3+2g(x)y \right) y^l   dy \\
&=  \left( \frac{4}{l+4} k x^{k-1} y^{l+4} + \frac{2}{l+2} \left( k x^{k-1} g(x) + x^k g'(x) \right) y^{l+2} \right) dx,
\end{align*}
one gets
\begin{align*}
&\mathrel{\phantom{=}} x^k \left( \frac{l}{l+2} g'(x) y^2 + h'(x) \right) y^l dx - kx^{k-1} \left( \frac{4}{l+4} y^4 + \frac{2}{l+2} g(x) y^2 \right) y^l dx \\
&= d(S_{l,k}), \ \mathrm{where} \  S_{l,k} := -  x^k \left( \frac{4}{l+4} y^{l+4}+\frac{2}{l+2} g(x) y^{l+2} \right)
\end{align*}
Using $y^4 = -( g(x)y^2+ h(x) )$ and the explicit description of $g(x)$ and $h(x)$ in (\ref{eq2}), the above
equation becomes
\begin{equation}
\left( \sum_{j=-1}^{3} \Gamma_{k,l,0,j} \,  x^{k+j} y^l + \sum_{j=-1}^{1} \Gamma_{k,l,2,j} \,
x^{k+j} y^{l+2} \right) dx = d(S_{l,k}), \label{eq3}
\end{equation}
here the coefficients $ \Gamma_{k,l,0,j} $ and $ \Gamma_{k,l,2,j} $ are defined as following
\begin{align*}
\Gamma_{k,l,0,j} := (j+1+\frac{4k}{l+4}) \, b_{j+1}, \   \ \Gamma_{k,l,2,j} := \frac{l}{l+2} ( j+1+
\frac{2k}{l+4} ) a_{j+1}.
\end{align*}
In order to make things more clear, we use the following notation:
\begin{definition} \label{def1}
A family of matrices $M_k$ of size $ m \times n $ with entries $(M_k)_{i,j} = m_{k,i,j} \in \mathbb{Q}_q$ is called
\emph{a family of reduction matrices} if
\[
\left( \sum_{i=1}^{m} \sum_{j=1}^{n}  m_{k,i,j} \,  x^{k+j-3} y^{i-1} \right) dx \equiv 0 \ \mathrm{in \ \Omega^1, \ for \  all \ } k.
\]
For reduction matrices $M_k$, we define $M_k \, dx$ to be
\[
\left( \sum_{i=1}^{m} \sum_{j=1}^{n}  m_{k,i,j} \,  x^{k+j-3} y^{i-1} \right) dx.
\]
\end{definition}

For example, from (\ref{eq3}), we have reduction matrices of size $ (l+3) \times 6 $  which has non-zero entries only
at the $(l+1)$-th and $(l+3)$-th rows
\begin{equation*} M_{l,k}^{0} := \scalemath{0.85} {
\left(
  \begin{array}{cccccc}
   0 & 0 & 0 & 0 & 0 & 0 \\
   \cdot & \cdot & \cdot & \cdot & \cdot & \cdot  \\
   \cdot & \cdot & \cdot & \cdot & \cdot & \cdot \\
   0 & 0 & 0 & 0 & 0 & 0 \\
   0 & \Gamma_{l,k,0,-1} & \Gamma_{l,k,0,0} & \Gamma_{l,k,0,1} & \Gamma_{l,k,0,2} & \Gamma_{l,k,0,3} \\
   0 & 0 & 0 & 0 & 0 & 0 \\
   0 & \Gamma_{l,k,2,-1} & \Gamma_{l,k,2,0} & \Gamma_{l,k,2,1} & 0 & 0 \\
  \end{array}
\right) } \label{eq4}
\end{equation*}
and $M_{l,k}^0dx = d(S_{l,k})$. The superscript 0 that appears in $M_{l,k}^0$ means that it is obtained
from (\ref{eq3}), without further reduction.

We have to consider $l = 1,2$ and $3$, which give the reductions of $ x^{k+3} y^2 dx$ or $ x^{k+2} y^2 dx$:
\begin{enumerate}[label={  \textbf{\textit{l}} $\mathbf{\boldsymbol{=}\hspace{-3pt}~\arabic*.}$ }, ref={  $l \ \mbox{=}\hspace{-1.2pt}~\arabic*$ }, labelwidth=\dimexpr-\mylen-\labelsep\relax, labelsep=5mm, leftmargin=0mm]
\item  \  \label{case2} \\
As mentioned above, we have reduction matrices of size $4 \times 6$
\[
M_{1,k} = \frac{1}{15} \cdot \scalemath{0.85} { \left(
            \begin{array}{cccccccc}
             0 & 0 & 0 & 0 & 0 & 0  \\
             0 & 12k b_0 & (12k+15) b_1 & (12k+30) b_2 & (12k+45)  b_3 & (12k+60) b_4  \\
             0 & 0 & 0 & 0 & 0 & 0   \\
             0 & 2k  a_0 & (2k+5)  a_1 & (2k+10)  a_2 & 0 & 0   \\
            \end{array}
          \right) }
\]
From \ref{Geometry}, we know that one of $b_4$ and $b_3$ is non-zero. Hence $x^{k+3} y dx$ or
$x^{k+2} y dx$ can be reduced to a linear combination of $ \{ x^j y^i dx \} $ with $i=1$ or $3$ and $j
\leq k+2$ (if $b_4 \neq 0$) or $j \leq k+1$ (if $b_4=0$), which have smaller degree in $x$.

\item \  \label{case1} \\
From (\ref{eq4}), we have reduction matrices of size $5 \times 6$, so $y^4$ is involved. Using the defining
equation (\ref{eq1}) to reduce the degree in $y$, one gets a reduction matrix $M_{2,k}$ whose transpose is
\[
M^{t}_{2,k} = -\frac{1} {6} \cdot \scalemath{ 0.85 } {
\left(
  \begin{array}{cccc}
    0 & 0 & 0 & 0  \\
    0 & 0 & k(a_0^2-4b_0) & 0 \\
    0 & 0 & (2k+3)(a_0a_1-2b_1) & 0 \\
    0 & 0 & (k+3)(a_1^2+2a_0a_2-4b_2) & 0 \\
    0 & 0 & (2k+9)(a_1 a_2 - 2b_3) & 0 \\
    0 & 0 & (k+6)(a_2^2-4b_4) & 0 \\
  \end{array}
\right) }
\]
This gives us the reductions of $x^{k+3}y^2dx$ or $x^{k+2}y^2dx$ depending on the nullity of $a_2^2-4b_4$.
\item \  \\
As at the start of \ref{case1} (dealing with $y^5$), one gets $M^{1}_{3,k}$ whose transpose is
\[ \scalemath{0.85} {
-\frac{1} {35} \cdot
\left(
  \begin{array}{cccc}
   0 & 0 & 0 & 0  \\
   0 & 6ka_0b_0 & 0 & k(6a_0^2-20b_0)  \\
   0 &  6k(a_1b_0+a_0b_1)+21a_1b_0 & 0 & (4k+7)(3a_0a_1-5b_1)  \\
   0 & 6k(a_2b_0+a_1b_1+a_0b_2)+21a_1b_1+42a_2b_0 & 0 & (2k+7)(3a_1^2+6a_0a_2-10b_2)  \\
   0 & 6k(a_2b_1+a_1b_2+a_0b_3)+21a_1b_2+42a_2b_1 & 0 & (4k+21)(3a_1 a_2 - 5b_3)  \\
   0 & 6k(a_2b_2+a_1b_3+a_0b_4)+21a_1b_3+42a_2b_2 & 0 & (2k+14)(3a_2^2-10b_4)  \\
   0 & (6k+21)(a_2b_3+a_1b_4)+21a_2b_3 & 0 & 0  \\
   0 & (6k+42)a_2b_4 & 0 & 0  \\
  \end{array}
\right) }
\]
Since we want to reduce $x^{j} y^3 dx$ to those with smaller degree in $x$, $M^{1}_{3,k}$ is not suitable since
it has (possible) non-zero entries which correspond to $ x^{k+5} y^1$, $ x^{k+4} y^1$ and $ x^{k+3} y^1$. We
use \ref{case2} to reduce $x^{k+3} y dx$ to $\{  x^{j_1} y dx,  x^{j_2} y^3 dx \, | \, j_1 \leq k+2, j_2 \leq
k+1\}$ if $b_4 \neq 0$, or reduce $x^{k+2} y dx$ to $\{  x^{j_1} y dx,  x^{j_2} y^3 dx\, | \, j_1 \leq k+1, j_2
\leq k+1\}$ ) if $b_4 = 0$. Then use this result to reduce $x^{k+3} y dx$ (if $b_4 = 0$), $x^{k+4} y dx$ and $x^{k+5} y dx$ (if $b_4 \neq 0$)
iteratively to $\{  x^{j_1} y dx,  x^{j_2} y^3 dx\, | \, j_1 \leq k+2 , j_2 \leq k+3\}$,
depending the nullity of $b_4$. Finally, use these reductions in $M^{1}_{3,k}$, we get reduction
matrices in
\begin{enumerate}[label={ }, ref={Case 3}, labelwidth=\dimexpr-\mylen-\labelsep\relax, labelsep=0pt , leftmargin=0mm]
\item \hspace*{-54pt} \ref{case3} and \ref{case7}:  $b_4 \neq 0$ \label{case4}
\[
M^{'}_{3,k}= c \cdot \scalemath{0.85} { \left(
  \begin{array}{cccccc}
   0 & 0 & 0 & 0 & 0 & 0 \\
   0 &  \ast & \ast & \ast & \ast_{2,5} & 0 \\
   0 & 0 & 0 & 0 & 0 & 0 \\
   0 & \ast & \ast & \ast &  \ast_{4,5} & 384(k+4)(k+5)(k+6)(k+7)b_4^3(a_2^2-4b_4) \\
  \end{array}
\right) }
\]
with  $ c = -1 / \left( 2688 (k+4) (k+5) (k+6) b_4^3 \right) $ and
\[
\ast_{4,5} =  96 b_4^2 (k+4)(k+5)(k+6) \left(   (8k+44)a_1a_2b_4-a_2^2b_3-(16k+84)b_3b_4 \right)
\]
which equals to $ 384 b_4^3 (k+4) (k+5) (k+6) (2k+11)  (a_1a_2-2b_3)$ if $a_2^2-4b_4 = 0$, and
% \begin{equation*}
\begin{align*} \ast_{2,5} = 36 (k+4) b_4
&\mathrel{\phantom{}}   \bigg( (6k^2+136k+285)a_2b_3^3-8(6k^2+56k+125)a_2b_2b_3 \\
&\mathrel{\phantom{}} \mathrel{\phantom{=}} + \, 48(k^2+11k+30)a_2b_1b_4^2 +16(k^2+6k)a_0b_3b_4^2 \\
&\mathrel{\phantom{}} \mathrel{\phantom{=}} - \, 2(8k^2+58k+105) a_1b_3^2b_4 + 16(2k^2+17k+35)a_1b_2b_4^2 \bigg)
\end{align*}
% \end{equation*}
Notice that $\ast_{2,5}$ may be non-zero, but it corresponds to $x^{k+2} y$ and since $b_4 \neq
0$, one can use $M_{1,k-1}$ in the case \ref{case2} to reduce $x^{k+2} y dx$ and get a new
reduction matrices
\[
M_{3,k}= c \cdot \scalemath{0.85} { \left(
  \begin{array}{cccccc}
   0 &  0 & 0 & 0 & 0 & 0 \\
   \ast &  \ast & \ast & \ast & 0 & 0 \\
   0 & 0 & 0 & 0 & 0 & 0 \\
   \ast & \ast & \ast & \ast &  \ast_{4,5} & 384(k+4)(k+5)(k+6)(k+7)b_4^3(a_2^2-4b_4) \\
  \end{array}
\right) }
\]
Notice that the reduction of $x^{k+2} y dx$ using $M_{1,k-1}$ only involve $x^{j} y dx$ with $ k-2 \leq j \leq
k+1$ and $x^{j} y^3$ with $ k-2 \leq j \leq k$, the last two columns of $M_{3,k}$ and $M^{2}_{3,k}$ are the
same except the $(2,5)$-entry, and $M_{3,k}$ satisfies the condition in Definition \ref{def1}, so they are
indeed reduction matrices. The reduction involves division by $12(k+4)b_4$, hence $c^{-1} M_{3,k}$ has
integral coefficients.

\item \hspace*{-54pt} \ref{case5} and \ref{case6}: $b_4 = 0$ \\
\[ \scalemath{0.85} {
M_{3,k}= c \cdot \left(
  \begin{array}{cccccc}
   0 &  0 & 0 & 0 & 0 & 0 \\
   0 & \ast & \ast & \ast & 0 & 0 \\
   0 &  0 & 0 & 0 & 0 & 0 \\
   0 & \ast & \ast & \ast &  \ast_{4,5} & 2(k+7)(2k+11)(4k+15)(4k+19) b_3^2 a_2^2 \\
  \end{array}
\right) }
\]
with $c = -1 / \left( 7(4k+15)(4k+19)(4k+23)b_3^2 \right)$ and
% \begin{equation*}
\begin{align*} \ast_{4,5} =(4k+15)b_3
&\mathrel{\phantom{}}   \bigg( (32k^3+504k^2+2648k+4641) a_1a_2b_3 - (4k^2+52k+168)a_2^2b_2 \\
&\mathrel{\phantom{}} \mathrel{\phantom{=}} - \, (64k^3+1008k^2+5276k+9177)b_3^2  \bigg)
\end{align*}
% \end{equation*}
which equals to $-(4k+15)(4k+19)(4k+21)(4k+23) b_3^3$ if $a_2 = 0$.
\end{enumerate}
\end{enumerate}

Now we can compute the algebraic de Rham cohomology $ H^{1}_{\scriptscriptstyle{dR}} ( \tilde{C}_{\xi, \mathrm{aff}} /
\mathbb{Q}_q ) $ of $\tilde{C}_{\xi, \mathrm{aff}} / \mathbb{Q}_q $.

\begin{prop} \label{prop1}
The algebraic de Rham cohomology $H^{1}_{\scriptscriptstyle{dR}} ( \tilde{C}_{\xi, \mathrm{aff}} / \mathbb{Q}_q )$ has a
basis
  \begin{enumerate}[topsep=1.5mm, label=\textup{(\arabic{*})}] \setlength{\itemsep}{-0.0mm}
\item $\{ ydx, y^2dx, y^3dx, xydx, xy^2dx, xy^3dx \} $, if $b_4=0 $ and $a_2 = 0$.
\item $\{ ydx, y^2dx, y^3dx, xydx, xy^2dx, xy^3dx, x^2y^2dx, x^2y^3dx \}$, if $ b_4 = 0$ and \\ $a_2 \neq
    0$.
\item $ \{ ydx, y^2dx, y^3dx, xydx, xy^2dx, xy^3dx, x^2 y dx \}$, if $ b_4 \neq 0 $ and $a_2^2-4b_4 =
    0$.
\item $ \{ ydx, y^2dx, y^3dx, xydx, xy^2dx, xy^3dx, x^2ydx, x^2y^2dx, x^2y^3dx \}$, if $ b_4 \neq 0$
    and $a_2^2-4b_4 \neq 0$.
\end{enumerate}
\end{prop}

\begin{proof}
We give the proof for (c). For other parts, the proofs are all similar. Suppose $ b_4 \neq 0$ and
$a_2^2-4b_4 = 0$. The reduction matrices $M_{1,k}$ shows that $x^{k+3} y dx$ is a linear combination of
$\{x^{j_1} y dx, x^{j_2} y^3 dx \ | \ k-1 \leq j_1 \leq k+2, k-1 \leq j_2 \leq k+1   \}$, since $b_4 \neq
0$. So each $x^j y dx$ with $j \geq 3$ can be reduced. The $(3,6)$-entry of the reduction matrix $M_{2,k}$
is $(k+6)(a_2^2-4b_4) = 0$, but its $(3,5)$-entry is $(2k+9) ( a_1a_2-2b_3)$ which is non-zero by
\ref{case3} in Section \ref{Geometry}. So $x^{k+2} y^2 dx$ is a linear combination of $ \{ x^{j} y^2 \ | \
k-1 \leq j \leq k+1 \}$ and hence each $x^j y^2 dx$ with $j \geq 2$ can be reduced. The reduction matrix
$M_{3,k}$ in \ref{case4} has $(k+7)(a_2^2- 4b_4) = 0$ at the $(4,6)$-entry, but its $(4,5)$-entry is
$\ast_{4,5} = (2k+11)(a_1a_2-2b_3) \neq 0$. So $x^{k+2} y^3 dx$ is a linear combination of $ \{
x^{j_1} y dx, x^{j_2} y^3 dx \ | \ k-2 \leq j_1 \leq k+1, k-2 \leq j_2 \leq k +1 \}$ and hence each
$x^{j} y^3$ with $j \geq 2$ can be reduced. This completes the proof for (c).
\end{proof}

The following table give a more clear description of these basis.
\begin{table}[ht]
\renewcommand\arraystretch{1.39}
\begin{subtable}{2.75cm}
\centering
\begin{equation*}  \scalemath{0.8}{
\begin{array}{ | c | c | c | c | } \hline
  dx & 1 & x & x^2 \\ \hline
  y & \bullet & \bullet & \times \\ \hline
  y^2 & \bullet & \bullet & \times \\ \hline
  y^3 & \bullet & \bullet & \times \\ \hline
\end{array} }
\end{equation*}
\subcaption{\ref{case5}}
  \label{table1}
  \end{subtable}
\begin{subtable}{2.75cm}
  \centering
\begin{equation*} \scalemath{0.8}{
\begin{array}{ | c | c | c | c | } \hline
  dx & 1 & x & x^2 \\ \hline
  y & \bullet & \bullet & \times \\ \hline
  y^2 & \bullet & \bullet & \bullet \\ \hline
  y^3 & \bullet & \bullet & \bullet \\ \hline
\end{array}   }
\end{equation*}
  \subcaption{\ref{case6}}
  \label{table2}
  \end{subtable}
    \begin{subtable}{2.75cm}
  \centering
\begin{equation*} \scalemath{0.8}{
\begin{array}{ | c | c | c | c | } \hline
  dx & 1 & x & x^2 \\ \hline
  y & \bullet & \bullet & \bullet \\ \hline
  y^2 & \bullet & \bullet & \times \\ \hline
  y^3 & \bullet & \bullet & \times \\ \hline
\end{array}     }
\end{equation*}
  \subcaption{\ref{case3}}
  \label{table4}
  \end{subtable}
\begin{subtable}{2.75cm}
  \centering
 \begin{equation*} \scalemath{0.8}{
\begin{array}{ | c | c | c | c | } \hline
  dx & 1 & x & x^2 \\ \hline
  y & \bullet & \bullet & \bullet \\ \hline
  y^2 & \bullet & \bullet & \bullet \\ \hline
  y^3 & \bullet & \bullet & \bullet \\ \hline
\end{array}      }
\end{equation*}
  \subcaption{\ref{case7}}
  \label{table3}
  \end{subtable}
  \caption{basis of $H^{1}_{\scriptscriptstyle{dR}} ( \tilde{C}_{\xi, \mathrm{aff}} / \mathbb{Q}_q )$}
  \label{table5}
\end{table}

\subsection{Control of the denominators in the reduction algorithm and Monsky-Washnitzer cohomology} \label{control}

The reduction algorithm in subsection \ref{Reduction} allows us to obtain a basis of $H^{1}_{\scriptscriptstyle{dR}} ( \tilde{C}_{\xi, \mathrm{aff}} /
\mathbb{Q}_q )$. By Theorem \ref{thm1}, this basis also forms a basis of the Monsky-Washnitzer cohomology
$H_{\scriptscriptstyle{MW}}^1(C_{\mathrm{aff}}/\mathbb{F}_q)$. One can also prove this by the following upper
bound on the denominators that appear during the reduction process. This bound provides the precision necessary for our algorithm.

Before stating the main result of this subsection, we fix some notations. For a local parameter $t$ at a
point at infinity $P_{\infty}$, we write the Laurent series expansion of $x$, $y$ and $ x^j y^i$ with
respect to $t$ as following:
\begin{equation}
x = \sum^{\infty}_{s = \upsilon_{p}(x)}  \delta^{0,1}_s \, t^s, \ \
y = \sum^{\infty}_{s = \upsilon_{p}(y)} \delta^{1,0}_s \, t^s \ \ \mathrm{and} \ \
x^j y^i = \sum^{\infty}_{s = \upsilon_{p}(x^j y^i)} \delta^{i,j}_s \, t^s.
\end{equation}
If a subscript is used to denote a local parameter at some point, we use this subscript in the
coefficients of the above expansion. For example, in \ref{case3}, we write $x^j y^i = \sum^{\infty}_s
\delta^{i,j}_{s,\scriptscriptstyle{+}} \, t^s_{\scriptscriptstyle{+}}$  at $P_{\infty,+}$ and $x^j y^i =
\sum^{\infty}_s \delta^{i,j}_{s,\scriptscriptstyle{-}} \, t^s_{\scriptscriptstyle{-}}$  at
$P_{\infty,-}$. Recall that all the coefficents $\delta^{i,j}_s$ are in $\mathcal{O}$ in all cases that
we are concerned, see Remark \ref{rem1}.

\begin{prop} \label{prop2}
Write
\begin{equation} \label{eq5}
 x^k y^l dx = \sum^{3}_{ i=1 } \sum^{2}_{ j=0 } a_{i,j}  x^j y^i dx + dS, \mathrm{ \ where \  } S = \sum^{3}_{i=0} \sum_{j \geq 0} b_{i,j}  x^j y^i
\end{equation}
with $a_{i,j}$ and $b_{i,j} \in \mathbb{Q}_q$, $a_{i,j} = 0$ if $ x^j y^i$ is not in the basis in
Proposition \ref{prop1}, $ 1 \leq l \leq 3$ and $k \in \mathbb{N}$. Then
\begin{enumerate}[topsep=1.5mm, ref={\alph*}] \setlength{\itemsep}{-0.0mm}
\item One can choose $S$ with the property that $b_{i,j} = 0$ if $i-l \neq 0 \ (\mathrm{mod} \ 2)$.
\item \label{bb} For any $S$ in $\mathrm{(\ref{eq5})}$ which satisfies the property in {\rm (a)}, we have $b_{i,j-i} = 0$
    for all $0 \leq i \leq 3$ and $j \geq k+5$. Furthermore $p^m b_{i, j-i} \in \mathcal{O}$ for all $0 \leq i \leq 3$
    and $j \geq 7$, where $m = \lfloor \mathrm{log}_p (4k+8) \rfloor$.
\item $p^{m + \Delta+1} a_{i,j} \in \mathbb{Z}_q $, where $ \Delta := 11 \left( \lfloor \mathrm{ log }_p ( 63 ) \rfloor + \tau(p) \right)
    $ with $\tau(3) = 5, \tau(5) =3, \tau (p) = 1$ for $ p= 7,11,13$ and $\tau(p) = 0$ if $p> 13$.
\end{enumerate}
\end{prop}

\begin{proof}
(a) Using the automorphism $y \rightarrow -y$, or by a direct analysis on the reduction process
    discussed in the previous subsection.
(b) We prove this for \ref{case7}, proofs for other cases are all similar. From the expansions in \ref{case7}, one obtains
\begin{equation}
 x^j y^i = \sum^{\infty}_{s=-(i+j)} \delta^{i,j}_{s,\mu,\pm} t^{s}_{\mu,\pm}
\end{equation}
with $\delta^{i,j}_{-(i+j),\mu,\pm} = ( \pm \alpha_{\mu})^{i}$ and $\mu = 1,2$. There is an integer $M > 0$
such that $b_{i,j} = 0$ for all $j > M$. Hence
\begin{align*}
S &= \sum^{3}_{i=0} \sum^{M}_{j \geq 0} b_{i,j}  x^j y^i = \sum^{3}_{i=0} \sum^{M}_{j \geq 0} b_{i,j} \sum^{\infty}_{s=-(i+j)} \delta^{i,j}_{s,\mu,\pm} t_{\mu,\pm}^s \\
  &= \sum^{M+3}_{j=0} \left( \sum^{3}_{i=0} b_{i,j-i}  \delta^{i,j-i}_{-j,\mu,\pm} + \sum^{M+3}_{j' > j} \left( \sum^{3}_{i=0} b_{i,j'-i}  \delta^{i,j'-i}_{-j,\mu,\pm} \right) \right) t_{\mu,\pm}^{-j}
\end{align*}
Since $\upsilon_{p_{\infty,\mu, \pm}}(a_{i,j}  x^j y^i dx) \geq 7$ and $\upsilon_{p_{\infty, \mu, \pm}}( x^k y^l
dx) \geq -(k+5)$ and the expansions of $ x^k y^l dx$ have integral coefficients, we have
\begin{equation} \label{eq10}
j \cdot \left( \sum^{3}_{i=0} b_{i,j-i}  \delta^{i,j-i}_{-j,\mu,\pm} + \sum^{M+3}_{j' > j} \left( \sum^{3}_{i=0} b_{i,j'-i}  \delta^{i,j'-i}_{-j,\mu,\pm} \right) \right) \in \mathcal{O}
\end{equation}
for all $j \geq 7$ and it is zero if $ j \geq k+5$. Combine (\ref{eq10}) with the property in (a) and
the fact that $\alpha_1$ and $\alpha_2$ are units in $\mathcal{O}$ (Fact \ref{rem1}), we get
\begin{equation*}
j \cdot \left( 1 \cdot b_{i,j-i} + \alpha^2_{\mu} \cdot b_{i+2,j-i-2} + \sum^{M+3}_{j' > j} \sum^{3}_{i=0} \ast \right) \in \mathcal{O}
\end{equation*}
for $i=0,1$, $j \geq 7$ and $\mu = 1,2$, and it is zero if $j \geq k+5$. Since $\alpha_1^2 - \alpha_2^2 = \sqrt{D}$,
here $D = a_2^2 - 4 b_4 \neq 0$, one obtains
\begin{equation} \label{eq11}
j \cdot \left( 1 \cdot b_{i,j-i} + \sqrt{D} \cdot b_{i+2,j-i-2} + \sum^{M+3}_{j' > j} \sum^{3}_{i=0} \ast \right) \in \mathcal{O}
\end{equation}
for $i=0,1$, $j \geq 7$, and it is zero if $j \geq k+5$, here $\ast$ involves only $b_{i,j'-i}$ with $j' > j$ and
elements in $\mathcal{O}$. Remember that $\sqrt{D} \in \mathcal{O}^{\ast}$. Apply $j = M+3$ to (\ref{eq11}),
we know that $b_{i,M+3-i} = 0$ for all $0 \leq i \leq 3$. Repeat the same argument, one shows that $b_{i,j-i} = 0$
for all $0 \leq i \leq 3$ and $j \geq k+5$. Now apply $j=k+4$ to (\ref{eq11}), we get $p^m b_{i,k+4-i} \in
\mathcal{O}$ for all $0 \leq i \leq 3$. Repeat the same argument and notice that the terms $\ast$ in (\ref{eq11}) are
in $\mathcal{O}$ in each step (since all the $\delta^{\bullet,\bullet}_{s}$ and $b_{i,j'-i} \in \mathcal{O}$ if $j' >
j$ in each step),one proves that $p^m b_{i,j-i} \in \mathcal{O}$ for all $0 \leq i \leq 3$ and $j \geq 7$. So $p^m
b_{i,j-i} \in \mathcal{O} \cap \mathbb{Q}_q = \mathbb{Z}_q$ for all $0 \leq i \leq 3$ and $j \geq 7$.

\noindent(c) Consider
\begin{subequations}
\begin{align}
\omega &:= p^m \left(  x^k y^l dx - d \left(  \sum^{3}_{i=0} \sum^{k+5}_{j \geq 7-i } b_{i,j}  x^j y^i \right) \right)   \label{aa} \\
&\phantom{:}= p^m \left( \sum^{3}_{i=1} \sum^{2}_{j=0} a_{i,j}  x^j y^i dx + d \left( \sum^{3}_{i=0} \sum^{6-i}_{j=0} b_{i,j}  x^j y^i  \right) \right) \label{bbb}
\end{align}
\end{subequations}
From (\ref{bb}) and (\ref{aa}), one knows that $\omega$ has integral coefficients, so we can choose $\varphi_{1} (x,y)$ and
$\psi_{1} (x,y)$ in $\mathbb{Z}_q[x,y]$ such that $\omega =  \psi_1 \, dx + \varphi_1 \, dy $. On the other hand,
from (\ref{bbb}), one know that $\omega =   \psi_2 \, dx + \varphi_2 \, dy$ for some $\varphi_2$ and $\psi_{2}$ in
$\mathbb{Q}_q[x,y]$ with $ \mathrm{deg}( \varphi_2 ) \leq 5 $, $\mathrm{deg}( \psi_2 ) \leq 5$. Consider
\begin{equation}
\begin{aligned}
f_{y} \, \omega &= f_{y} \left(  \psi_i \, dx +  \varphi_i \, dy \right) = (  \psi_i  f_{y} - \varphi_i  f_{x} ) \, dx \\
f_{x} \, \omega &= f_{x} \left(  \psi_i \, dx + \varphi_i \, dy  \right) = ( \varphi_i  f_{x} - \psi_i  f_{y}) \, dy
\end{aligned}
\end{equation}
(using $f_x \, dx + f_y \, dy =df = 0$). Let $\lambda_i (x,y):=\psi_i  f_{y} - \varphi_i  f_{x}$. It is clear that
$\lambda_1 \in \mathbb{Z}_q[x,y]$ and $\lambda_2 \in \mathbb{Q}_q[x,y]$ with $\mathrm{deg}(\lambda_2)
\leq 8$. Using the defining equation $f$ to reduce the degree of $y$ in $\lambda_i$, we get $f_y \, \omega =
\tilde{ \lambda }_i \, dx$ and $f_x \, \omega = - \tilde{ \lambda }_i \, dy$ with $ \tilde{ \lambda_1 } \in
\mathbb{Z}_q [x, y] $, $\mathrm{deg}_y ( \tilde{ \lambda_1 } ) \leq 3 $, $ \tilde { \lambda_2 } \in
\mathbb{Q}_q [x, y] $, $ \mathrm{deg} ( \tilde{ \lambda_2 } ) \leq 8$ and $ \mathrm{deg}_y ( \tilde{
\lambda_2 } ) \leq 3$. Since $ ( \tilde{ \lambda_1 } - \tilde{ \lambda_2 } ) dx = f_y \, \omega - f_y \, \omega = 0 $
and $ \mathrm{deg}_y (  \tilde{ \lambda_1 } - \tilde{ \lambda_2 } ) \leq 3 $, we have $  \tilde{ \lambda_1 } =
\tilde{ \lambda_2 } $. This means that $f_y \, \omega = \tilde{ \lambda } \, dx$ and $f_x \, \omega = - \tilde{
\lambda } \, dy$ with $\tilde{ \lambda } := \tilde{ \lambda }_1 = \tilde{ \lambda }_2 $ which is in
$\mathbb{Z}_q[x, y]$ of $\mathrm{deg} ( \tilde{ \lambda } ) \leq 8$. By Corollary \ref{cor1}, there exist
$\alpha$ and $\beta$ in $ \mathbb{Z}_q[x, y] $ with $\mathrm{deg} ( \alpha ) \leq 5$ and $\mathrm{deg} (
\beta ) \leq 5$ such that $ \alpha f_y + \beta f_x  = 1 $ in $A$. So $ \omega = ( \alpha f_y + \beta f_x ) \,
\omega = ( \alpha \tilde{ \lambda } ) dx - ( \beta \tilde{ \lambda } ) dy$. Notice that $ \omega \equiv
\sum^{3}_{i=1} \sum^{2}_{j=0} \, p^m a_{i,j} \,  x^j y^i dx$, we can use the reduction of $( \alpha \tilde{
\lambda } ) dx - ( \beta \tilde{ \lambda } ) dy$ to get the denominators of $p^m a_{i,j}$. Since $ \deg( \alpha
\tilde{ \lambda } ) \leq 13 $ and $ \deg( \beta \tilde{ \lambda } ) \leq 13$, we need only to know the
denominators of the final reductions of $ x^j y^i dx$ and $ x^j y^i dy$ with $ 0 \leq i+j \leq 13$. Using $ x^j  y^i
dy \equiv -j / (i+1)  x^{j-1} y^{i+1} dx$, the defining equation $f$, and $x^j dx \equiv 0$, we only need to
consider the reductions of $ x^j y^i dx$ with $ 1 \leq i \leq 3$ and $ i+j \leq 13$, but if $p=3$, the extra
denominator 3 should be counted.

The reduction of $ x^k y^l dx$ ($1 \leq l \leq 3$) using the reduction matrices $M_{i,j}$ in subsection \ref{Reduction} involve divisions by
some of the following: $12(k+2) b_4$, $(12k+21) b_3$, $(k+3)(a_2^2-4 b_4)$, $(2k+5)(a_1 a_2 - 2 b_3)$,
$384 (k+1)(k+2)(k+3)(k+4) b_4^3 (a_2^2-4 b_4)$, $384 (k+2)(k+3)(k+4)(2k+7) b_4^3 ( a_1a_2 - 2 b_3)$,
$2(k+4)(2k+5)(4k+3)(4k+7)b_3^2 a_2^2$, $-(4k+7)(4k+11)(4k+13)(4k+15)b_3^3$, depending on each case,
and the numbers $2, a_2, b_4, b_3, a_2^2 - 4b_4$ and $a_1 a_2 -2 b_3$ that we need to consider (depending on
each case) are units of $\mathbb{Z}_q$. So in each step, we get extra denominators which are at most
\begin{equation*}
p^{ \lfloor  \mathrm{ log }_p ( 4j+15 )  \rfloor + \tau(p) },
\end{equation*}
here $\tau(3) = 5, \tau(5) =3, \tau (p) = 1$ for $ p= 7,1,13$ and $\tau(p) = 0$ if $p > 13$.
Since we are concerned with $ 2 \leq j \leq 12$, we need at most 11 reduction steps, so the denominators of the
reductions of $x^j y^i dx$ with $ 1 \leq i \leq 3$ and $ i+j \leq 13$ are at most
\begin{equation*}
p^{ 11 \cdot (  \lfloor  \mathrm{ log }_p ( 63 )  \rfloor + \tau(p)) }.
\end{equation*}
Hence $p^{m+\Delta+1} a_{i,j} \in \mathbb{Z}_q$.
\end{proof}

\begin{rem} \label{rem2}
Proposition \ref{prop2} gives an upper bound for the denominators after a differential form (with integral coefficients) is reduced to the linear combination of the basis we found in Propsition \ref{prop1}. Along with the rate of convergence of the Frobenius $F_p$ (see Corollary \ref{cor2}), one can determine how much $p$-adic precision we need to work with (and determine an integer $N_3$ such that one can work with modulo $x^{N_3}$), see Section \ref{Algorithm}. But one needs an upper bound for all the denominators that will appear during the computation (in the reduction step) in order to know how much precision of the reduction matrices $M_{i,j}$ are required and to have an analysis of the bit complexity. It turns out that one has a similar bound as in Proposition \ref{prop2}. The proof is completely similar.

\end{rem}

\begin{thm} \label{thm2}
Let $R$ be a field or a discrete valuation ring and $\mathfrak{m}$ be the maximal ideal of $R$. Let $f_0, \ldots,
f_n \in R[x_1, \ldots, x_n]$ with $\mathrm{deg} f_i = d_i $ and define
\begin{equation*}
\rho = d_0 + \ldots + d_n - n -1.
\end{equation*}
Denote the homogenization of $f_i$ by $f^h_i$  for $i = 0, \ldots, n$. Assume that there is no point in $\mathbb{P}^n (\overline{R/\mathfrak{m}})$
satisfies $f^h_0 = f^h_1 =  \ldots = f^h_n = 0$.
Then there exist polynomials $g_0, \ldots, g_n \in R[x_1, \ldots, x_n]$ with $\mathrm{deg} \, g_i \leq \rho + 1 - d_i$ for
$ i = 0, \ldots, n$ such that
\begin{equation*}
\sum^{n}_{i=0} g_i f_i = 1.
\end{equation*}
\end{thm}

\begin{proof}
This appears as Theorem 2 in Denef-Vercauteren~\cite{D-V}.
\end{proof}

\begin{cor} \label{cor1}
There exist $\alpha$ and $\beta$ in $ \mathbb{Z}_q[x, y] $ with $\mathrm{deg} ( \alpha ) \leq 5$ and
$\mathrm{deg} ( \beta ) \leq 5$ such that $ \alpha f_y + \beta f_x  = 1 $ in $A$. Furthermore, one can find such $\alpha$ and $\beta$ such that $\alpha$ has only odd degrees in $y$ and $\beta$ has only even degrees in $y$.
\end{cor}

\begin{proof}
Apply \ref{thm2} to $R = \mathbb{Z}_q$, $f_0 = f$, $f_1 = f_y$ and $f_2 = f_x$.

Apply \ref{thm2} to $R = \mathbb{Z}_q$, $f_0 = f$, $f_1 = f_y$ and $f_2 = f_x$. If $\alpha$ and $\beta$ don't satisfy the last property, consider the equality $\alpha(x,-y)f_y(x,-y) + \beta(x,-y)f_x(x,-y) = 1$ in $A$. From (\ref{eq1}), it is clear that $f_y(x, -y) = -f_y (x,y)$ and $f_x (x,-y) = f_x(x,y)$. We thus have
\[
\left( \frac{ \alpha(x,y) - \alpha(x,-y) }{2} \right) f_y + \left( \frac{ \beta(x,y) + \beta(x, -y) }{2} \right) f_x = 1,
\]
 which completes the proof.
\end{proof}

\section{Lift of Frobenius} \label{Frobenius}

In this section, we describe a lift $F_p$ of the absolute Frobeninus endomorphism $ \overline{F}_p :
\overline{a} \rightarrow \overline{a}^p$ on the coordinate ring $\overline{A}$ of $\overline{C}_{\mathrm{aff}}$ to $A^{ \dag}
$. This means that $F_p$ is a $\mathbb{Z}_p$-algebra endomorphism on $A^{\dag}$ such that $\pi \circ F_p =
\overline{F}_p \circ \pi$, where $\pi$ is the reduction modulo $p$. The lift $F_q$ of the $q$-th Frobenius
endomorphism of $\overline{A}$ is $F_p^{n}$, hence one can work with $F_p$ for the purpose of computation.

Denote by $\sigma$ the $p$-th power Frobenius endomorphism on $\mathbb{F}_p$ and also its lift on
$\mathbb{Z}_p$. Any lift $F_p$ satisfies
\begin{equation*}
F_p(x) \equiv x^p \ \mathrm{mod} \ p, \ \ \ F_p(y) \equiv y^p \ \mathrm{mod} \ p, \ \ \  F_p( f(x,y) ) = 0.
\end{equation*}
From Corollary \ref{cor1}, we know that there exist $\alpha$ and $\beta$ in $\mathbb{Z}_q[x,y]$ such that $\alpha f_y +~\beta f_x = 1$.
Define $\delta_y := \alpha^p$, $\delta_x := \beta^p$ and consider the equation
\begin{equation} \label{eq13}
G(Z) := F_p( f(x,y) ) = f^{\sigma}( x^p + \delta_x Z, y^p + \delta_y Z ) = 0
\end{equation}
in $A^{\dag}[Z]$. Then $G(0) = f^{\sigma} (x^p, y^p) \equiv f^{\sigma}( x^{\sigma}, y^{\sigma} ) = 0 \
\mathrm{mod} \ p$. Also $G'(0) = f^{\sigma}_y(x^p,y^p) \delta_y + f^{\sigma}_x(x^p,y^p) \delta_x \equiv
f^{\sigma}_y(x^\sigma, y^\sigma) \delta_y + f^{\sigma}_x(x^\sigma,y^\sigma) \delta_x \equiv f^p_y \delta_y
+ f^p_x \delta_x = f^p_y \alpha^p + f^p_x \beta^p = (f_y \alpha + f_x \beta)^p = 1 \ \mathrm{mod}  \  p$.
Hence by Hensel's lemma, there is a unique solution of (\ref{eq13}) in $A^{\infty}$. Use the following proposition
and its corollary, this solution is in fact in $A^{\dag}$. In fact, Corollary \ref{cor2} below gives an explicit lower
bound on the rate of convergence, which allows us (together with Proposition \ref{prop1}) to work with a finite and
explicit $p$-adic precision.

\begin{lem} \label{lem1}
Let $H(Z) = \sum h_k(x) Z^k \in \mathbb{Z}_q[x][Z]$ and $\Delta_{1,k} = d_k := \mathrm{deg}(h_k)$.
Assume $h_0 (x) \equiv 0 \ \mathrm{mod} \ p $ and $h_1 (x) \equiv 1 \ \mathrm{mod} \ p $. Let  $  0 \leq
\Delta_{n,0} \leq \Delta_{n,1} \leq \cdots \leq \Delta_{n,j} \leq \cdots$ $(n \geq 1$, $j \geq 0)$ with $\delta_j =
\Delta_{1,j}$ be integers stastifies the following conditions:
\begin{enumerate}[topsep=1.5mm] \setlength{\itemsep}{0mm}
    \item[1.] $\Delta_{n+1,j} \geq \max \{ \Delta_{n,j-l} + \delta_l \ | \ 0 \leq l \leq j \}$ for all $n \geq 0, j
        \geq 0$.
    \item[2.] $\delta_0 \geq d_0$.
    \item[3.] $\delta_{j} - \delta_{j-1} \geq d_1$ for all $j \geq 1$.
    \item[4.] $\delta_{k-1+j} \geq  \Delta_{k,j} + d_k$ for all $k \geq 2, j \geq 0$.
\end{enumerate}
Then the unique solution $\alpha = \sum_{i=0}^{\infty} a_i x^i \in \mathbb{Z}_q \langle x \rangle$ has the
 property:  $v_p(a_i) \geq j+2 \ \mathrm{if} \ i \geq \delta_j + 1 $. One can always find such $\Delta_{n,j}$.

\end{lem}

\begin{proof}
Let $T := \{ \sum_{i=0}^{\infty} a_i x^i \in \mathbb{Z}_q \langle x \rangle \ | \ v_p(a_i) \geq j+2 \
 \mathrm{if} \ i \geq \delta_j + 1 \}$ be a subset of $\mathbb{Z}_q \langle x \rangle $, here $-1 = \delta_{-1} <
 0 \leq \delta_0 \leq \delta_1 \leq \cdots \leq \delta_j \leq \cdots$ are integers which we will determine for which
 $T$ satisfies some properties that are used in the proof. For each $n$, one can write $T^n = \{
 \sum_{i=0}^{\infty} a_i x^i \in \mathbb{Z}_q \langle x \rangle \ | \ v_p(a_i) \geq n+j+1 \ \mathrm{if} \ i \geq
 \Delta_{n,j} + 1 \}$ for some $-1 = \Delta_{n,-1} < 0 \leq \Delta_{n,0} \leq \Delta_{n,1} \leq \cdots \leq
 \Delta_{n,j} \leq \cdots$. A sufficient condition for $T^n \subset T$ is $\Delta_{n+1,j} \geq \max \{
 \Delta_{n,j-l} + \delta_l \ | \ 0 \leq l \leq j \}$ for all $n$. Notice that $T$ is closed under addition. In
 $\mathbb{Z}_q \langle x \rangle$, we can use Newton method:
 \begin{equation*}
 \alpha_{i+1} = \alpha_i - \frac{H(\alpha_i)}{H'(\alpha_i)} = \alpha_i - H(\alpha_i) \left(  1 + \sum^{\infty}_{k=1} \left( 1-H'(\alpha_i) \right)^k \right).
 \end{equation*}
 Our goal is to determine a condition on $\delta_j$ for which the result in each iteration above is in $T$. We use
 induction: assume $\alpha_i \in T$ and to prove $\alpha_{i+1} \in T$. It is sufficient to show: $H(\alpha_i) \in T$ and $ \left(1 - H'(\alpha_i) \right) T \subset T$.
$H(\alpha_i) = h_0(x) + h_1(x) \alpha_i + \sum_{k=2}^{\infty} h_k(x) \alpha_i^k$. Since $v_p ( h_0(x) ) \geq 1$, $h_0(x) \in T$ if $d_0 \leq \delta_0$.
Since $h_1(x) \alpha_i = \alpha_i - \left( 1 - h_1(x)  \right) \alpha_i$, $\alpha_i \in T$ and $v_p \left( 1-h_1(x) \right) \geq 1$, $h_1(x) \alpha_i \in T$ if $\delta_{j} - \delta_{j-1} \geq d_1$ for all $j \geq 1$.
Similarly for $k \geq 2$, since $\alpha_i^k \in T^k $, $h_k(x) \alpha_i^k \in T$ if $\Delta_{k,j} + d_k \leq \delta_{k-1+j}$. In fact, the above conditions imply $\left(1-h_1(x) \right) T \subset T$ and $h_k(x) T^k \subset T$.
Using this fact with $ \left( 1-H'( \alpha_i ) \right) T = \left( \left( 1-h_1(x) \right) + \sum_{k=2}^{\infty} -k h_k( x ) \alpha_i^{k-1} \right) T$ and $\alpha_i^{k-1} \in T^{k-1}$, we know $ \left( 1 - H'( \alpha_i ) \right) T \subset T$.
Hence $\alpha_{i+1} \in T$ and this implies the solution $\alpha \in T$.

For the existence $\Delta_{n,j}$, notice that the conditions 2, 3 and 4 are equivalent to: $\Delta_{1,n+1} \geq
 \max \{ \Delta_{1,n} + d_1, \Delta_{k,n+2-k} + d_k \   | \ 2 \leq k \leq n+2 \}$ for all $n \geq 0$. Suppose one
 has determined $\Delta_{n',j'}$ for all $n'+j' \leq n+1$. Use condition 1, one can determine $\Delta_{k,n+2-k}$
 for all $2 \leq k \leq n+2$ (i.e for $\Delta_{n',j'}$ with $n'+j' \leq n+2$ and $n' \geq 2$). Finally, one determines
 $\Delta_{1,n+1}$. Therefore, one determines all the $\Delta_{n',j'}$ with $n'+j' \leq n+2$. This shows that one
 can find $\Delta_{n,j}$ recursively.
\end{proof}

\begin{lem} \label{lem2}
Suppose $\mathrm{deg}(h_k(x)) \leq \left( k+1 \right) d$ in Lemma \ref{lem1}. Then $\Delta_{i,j} := \left( i+4j
\right) d$ satisfy the conditions in Lemma \ref{lem1} and all the inequalities are equalities. In particular, $\delta_j =
\left( 4j+1 \right)d$.
\end{lem}

\begin{proof}
This follows by induction in $i$ and $j$.
\end{proof}

\begin{cor} \label{cor2}
There exists a lift $F_p$ of the absolute Frobeninus endomorphism $ \overline{a} \rightarrow \overline{a}^p$ on
the coordinate ring $\overline{A}$ of $\overline{C}_{\mathrm{aff}}$ to $A^{ \dag}$ such that $F_p(x) = x^p + \delta_x Z_0$ and $F_p(y) = y^p + \delta_y Z_0$ with $Z_0 = \sum_{i,j} a_{i,j}  x^j y^i$, $a_{i,j} \in \mathbb{Z}_q$ and
$\mathrm{ord}_p ( a_{i,j} ) >\mathrm{} \frac{i+j}{16p}$. Also the coefficient of $ x^j y^i$ in $F_p(y)$ and $
F_p(x) $ has $p$-adic order $ > \frac{i+j}{16p} $ if $i+j \neq p$. Finally, $F_p(x^k y^l dx) = \sum_{i=1}^3
\sum_j b_{i,j} x^j y^i dx$ with $\mathrm{ord}_p (b_{i,j}) > \frac{i+j}{16p} - 4$.
\end{cor}

\begin{proof}
Using Corollary \ref{cor1} and equation (\ref{eq13}), one can apply $d = 4p$ to Lemma \ref{lem2}.
\end{proof}

\begin{thm}
 There exists a lift of Frobenius $F_p$ on $A^{\dag}$ which commutes with the involution $\tau: y \rightarrow -y$ and has the rate of convergence in Corollary \ref{cor2}.
\end{thm}

\begin{proof}
We choose $\alpha$ and $\beta$ such that $\alpha$ has only odd degrees in $y$ and $\beta$ has only even degrees in $y$ as in Corollary \ref{cor1}. Since $\delta_y = \alpha^p $ and $\delta_x = \beta^p$, they have the same property as $\alpha$ and $\beta$. For solving $G(Z) = 0$ by Newton's method, we use Lemma~\ref{lem1} with $H = G = f^{\sigma}( x^p + \delta_x Z, y^p + \delta_y Z )$. It is clear that $G$ has only even degrees in $y$, hence so does the solution $Z_0$. From this, it is clear that the lift of Frobenius $F_p: A^{\dag} \rightarrow A^{\dag}$ commutes with the involution~$\tau$.
\end{proof}

\section{Quotient by Automorphism} \label{quotient}

We have study $H_{\scriptscriptstyle{MW}}^1(C_{\mathrm{aff}}/\mathbb{F}_q)$. In this section, we consider
the quotient of $C$ by the automorphism $ \tau : Y \rightarrow -Y$. We denote the quotient map by $ \pi : C
\rightarrow E:= C / \langle \tau \rangle$. One can show that $C / \langle \tau \rangle$ has genus 1 either by
Riemann-Hurwitz genus formula or from the affine equation directly, using the fact that $C_{\mathrm{aff}}$ is
stable under $\tau$ and $C_{\mathrm{aff}} / \langle \tau \rangle$ is smooth, hence the notation $E$ is justified.
The affine part $E_{\mathrm{aff}}$ of $E$ is $C_{\mathrm{aff}} / \langle \tau \rangle$, which has the defining
equation: $v^2 + \overline{g}(u)v + \overline{h}(u) = 0$. We have $C_{\mathrm{aff}} \xrightarrow{\pi}
E_{\mathrm{aff}}$, $(x,y) \rightarrow (x,y^2)$, and the corresponding map on the coordinate ring is $
\pi^{\ast}: u \rightarrow x$, $v \rightarrow y^2$.

Our goal is to study the followings: $H_{\scriptscriptstyle{MW}}^1(E_{\mathrm{aff}}/\mathbb{F}_q)$, the
induced map $\pi^{\ast}: H_{\scriptscriptstyle{MW}}^i(E_{\mathrm{aff}}/\mathbb{F}_q) \rightarrow
H_{\scriptscriptstyle{MW}}^i(C_{\mathrm{aff}}/\mathbb{F}_q)$ and  its interplay with Frobenius
endomorphism. Since $\tilde{C}_{\mathrm{aff}} \xrightarrow{\tilde{\pi}} \tilde{E}_{\mathrm{aff}}$, $(x,y)
\rightarrow (x,y^2)$ lifts $\pi$, here $\tilde{E}_{\mathrm{aff}}$ is the lift of $E_{\mathrm{aff}}$, whose
defining equation is $v^2 + g(u)v + h(u) = 0$, we can study
$H_{\scriptscriptstyle{MW}}^i(E_{\mathrm{aff}}/\mathbb{F}_q) \xrightarrow{\pi^{\ast}}
H_{\scriptscriptstyle{MW}}^i(C_{\mathrm{aff}}/\mathbb{F}_q)$ by
$H_{\scriptscriptstyle{dR}}^i(\tilde{E}_{\mathrm{aff}, \xi}/\mathbb{Q}_q) \xrightarrow{\tilde{\pi}^{\ast}}
H_{\scriptscriptstyle{dR}}^i(\tilde{C}_{\mathrm{aff}, \xi}/\mathbb{Q}_q)$. For $i \neq 1$, these are
isomorphisms. For $i = 1$, since $ \{ \tilde{\pi}^{\ast}(u^j v du) =  x^j y^2 dx \} \ | \  0 \leq  j \leq 1 \ ({\rm
resp.} \  0 \leq  j \leq 2 ) \}$ are linear independent in \ref{case5} and \ref{case3} (resp. in \ref{case6} and
\ref{case7}), one sees that $\{  u^j v du \  | \ 0 \leq j \leq 1 \ ( {\rm resp.} \ 0 \leq j \leq 2   )  \}$ are linear
independent. Let $\delta_E$ be the number of points at infinity of $E$. We have $\delta_E = 1,2,1,2$ in each case,
hence $ \mathrm{dim}_{\mathbb{Z}_q} H_{\scriptscriptstyle{dR}}^1(\tilde{E}_{\mathrm{aff},
\xi}/\mathbb{Q}_q) = 2 \cdot g_E - 1 + \delta_E = 2,3,2,3$. This shows that $\{  u^j v du \  | \ 0 \leq j \leq 1 \
( {\rm resp.} \ 0 \leq j \leq 2   )  \}$ is the basis of $H_{\scriptscriptstyle{dR}}^1(\tilde{E}_{\mathrm{aff},
\xi}/\mathbb{Q}_q)$ which is isomorphic via $ \tilde{\pi}^{\ast} $ to the subsapce $V$ of
$H_{\scriptscriptstyle{dR}}^1(\tilde{C}_{\mathrm{aff}, \xi}/\mathbb{Q}_q)$ generated by $\{   x^j y^2 dx \  |
\ 0 \leq j \leq 1 \ ( {\rm resp.} \ 0 \leq j \leq 2   )  \}$.

As in Section \ref{Frobenius}, there is a lift $F_{q,E} : A^{\dag}_E \rightarrow A^{\dag}_E$ of the Frobenius
endomorphism $\overline{F}_{q,E}$ on the coordinate ring $\overline{A}_E$ of $E$. The left diagram below is
not necessary commutative, but its reduction mod $p$ is commutative
\begin{equation*}
\xymatrix@!0{
A^{\dag} \ar[dd]_{F_q}  &  &  A_E^{\dag} \ar[dd]^{F_{q,E}} \ar[ll]_{\tilde{\pi}^{\ast}}  \\
 &  & \\
A^{\dag}  & &  A_E^{\dag} \ar[ll]^{\tilde{\pi}^{\ast}}
}
\ \ \ \ \ \ \ \ \ \
\xymatrix@!0{
\overline{A} \ar[dd]_{\overline{F}_q}  &  &  \overline{A}_E \ar[dd]^{\overline{F}_{q,E}} \ar[ll]_{\pi^{\ast}}  \\
 &  & \\
\overline{A}  & &  \overline{A}_E \ar[ll]^{\pi^{\ast}}
}
\end{equation*}
Here $A^{\dag} \xleftarrow{\hspace{5pt}\tilde{\pi}^{\ast}} A_E^{\dag}$ is the natural lift of
homomorphism $A \xleftarrow{\hspace{5pt}\tilde{\pi}^{\ast}} A_E$ on the coordinate rings which corresponds
to the morphism $\tilde{\pi} : \tilde{C}_{\mathrm{aff}, \xi} / \mathbb{Q}_q \rightarrow
\tilde{E}_{\mathrm{aff}, \xi} / \mathbb{Q}_q$, so the reduction of $A^{\dag} \xleftarrow{\hspace{5pt}\tilde{\pi}^{\ast}} A_E^{\dag}$
modulo $p$ is just the natural homomorphism on the coordinate rings of $C_{\mathrm{aff}} \xrightarrow{\pi}
E_{\mathrm{aff}}$. Since $\overline{F}_q \circ \pi^{\ast} = \pi^{\ast} \circ \overline{F}_{q,E}$, we know that
\begin{equation*} \scalemath{1.0} {
\xymatrix{
H^i_{\scriptscriptstyle{MW}} ( C_{ \mathrm{aff} } / \mathbb{F}_q) \ar[dd]_{F_{q, \ast}}  &  &  H^i_{\scriptscriptstyle{MW}} ( E_{ \mathrm{aff} } / \mathbb{F}_q) \ar[dd]^{F_{q,E,\ast}} \ar[ll]_{\pi^{\ast}}  \\
 &  & \\
H^i_{\scriptscriptstyle{MW}} ( C_{ \mathrm{aff} } / \mathbb{F}_q)   & &  H^i_{\scriptscriptstyle{MW}} ( E_{ \mathrm{aff} } /  \mathbb{F}_q)  \ar[ll]^{\pi^{\ast}}
} }
\end{equation*}
So the point counting on $E_{\mathrm{aff}}$ is the same as computing on the subsapce of $H^1_{MW} ( C_{ \mathrm{aff} } / \mathbb{F}_q)$ generated by
$\{   x^j y^2 dx \  | \ 0 \leq j \leq 1 \ ( {\rm resp.} \ 0 \leq j \leq 2   )  \}$.

From Lefschetz fixed point formula (Theorem \ref{thm3}), we have
%\begin{equation*}
\begin{align*}
\# C_{\mathrm{aff}} \left( \mathbb{F}_{q^r}  \right) &= \mathrm{Tr} \left(  (q F_{q,*}^{-1})^r | H^{0}_{MW} ( C_{\mathrm{aff}}   )  \right) - \mathrm{Tr} \left(  (q F_{q,*}^{-1})^r | H^{1}_{MW} ( C_{\mathrm{aff}}   )  \right) \\
\# E_{\mathrm{aff}} \left( \mathbb{F}_{q^r}  \right) &= \mathrm{Tr} \left(  (q F_{q,*}^{-1})^r | H^{0}_{MW} ( E_{\mathrm{aff}}   )  \right) - \mathrm{Tr} \left(  (q F_{q,*}^{-1})^r | H^{1}_{MW} ( E_{\mathrm{aff}}   )  \right)
\end{align*}
% \end{equation}
Let $P_E(X) = (X-\beta_1)(X-\beta_2)$ be the Weil polynomial of $E$ and $S_r(E):=\beta_1^r + \beta_2^r$. Then $\# E_{\mathrm{aff}} (\mathbb{F}_{q^r}) = q^r + 1 - S_r(E) - \delta_E$, here $\delta_E$ is the number of points at infinity of $E$.
Use $\# C_{\mathrm{aff}}( \mathbb{F}_{q^r} )$ = $\# C ( \mathbb{F}_{q^r} ) - \delta_C$, we get
\begin{equation*}
\# C ( \mathbb{F}_{q^r} ) = q^r + 1 -  S_r(E) - \mathrm{Tr} \left(  (q F_{q,*}^{-1})^r | V  \right) + ( \delta_C - \delta_E ),
\end{equation*}
here $V$ is the subsapce of
$H_{\scriptscriptstyle{MW}}^1(C_{\mathrm{aff}}/\mathbb{F}_q)$ generated by $\{  x^j y^2 dx \  |
\ 0 \leq j \leq 1 \ ( {\rm resp.} \ 0 \leq j \leq 2   )  \}$, whose dimension is $4+ \delta_C - \delta_E$. This implies that the Weil polynomial $P(X)$ of $C/\mathbb{F}_q$ equals to $P_E(X) Q_V(X) (X-1)^{-(\delta_C - \delta_E)} $, here $Q_V(X)$ is the characteristic polynomial of $q F_{q,*}^{-1}$ acts on $V$. The characteristic polynomial $P_V(X)$ of $F_{q,*} = q \cdot (q F_{q,*}^{-1})^{-1}$
is $Q_V(X) (X-q)^{\delta_C- \delta_E} (X-1)^{-(\delta_C- \delta_E)}$, so
\begin{equation*}
P(X) = P_E(X) P_V(X) (X-q)^{-(\delta_C - \delta_E)}.
\end{equation*}

\begin{summary}
The Weil polynomial $P(X)$ of $C/\mathbb{F}_q$ is equal to
\begin{equation*}
P_E(X) P_V(X) (X-q)^{-(\delta_C - \delta_E)},
\end{equation*}
where $P_E(X)$ is the Weil polynomial  of $E/\mathbb{F}_q$ and $P_V(X)$ is the characteristic polynomial of $F_{q,\ast}$ on $V$.
\end{summary}

\section{The Algorithm} \label{Algorithm}
In order to compute $P_V(X)$, one needs to compute $P_V(X)$  with a precision $N_1 = \lfloor \mathrm{log}_p 30 +
2 n \rfloor + 1 $ with $n = \mathrm{log}_p q$, which is determined by the Weil bound.  Due to the fact that the
matrix $M_p$ of the Frobenius action $F_{p,\ast}$ may have denominators, we need $M_p$ with a precision $N_2:=
N_1 + (6 n - 1) c $ with $c = \lfloor c_1 +\mathrm{log}_p( c_1 + \mathrm{log}_p (2c_1) ) \rfloor + 1$ and $c_1 =
6 + \mathrm{log}_p 80 + \Delta$. From this, we only need to compute (for $ 1 \leq l \leq 3$ and  $0 \leq k \leq 2 $)
$Z_0$, $F_p(x)$, $F_p(y)$ and $F_p(y^l x^k dx)$ modulo $(x^{N_3}, p^{N_4})$ with $N_3 = \lfloor 16p ( c_2 +
\mathrm{log}_p ( 2 c_2 ) ) \rfloor + 1 , N_4 = \lfloor N_2 + c_1 + \mathrm{log}_p ( c_2 + \mathrm{log}_p (
2c_2)  ) \rfloor + 1 $ and $c_2 = 6 + \mathrm{log}_p 80 + \Delta + N_2$. Finally, the above discussion is based on
the reduction matrices $M_{i,k}$ $( 1 \leq i \leq 3)$ introduced in \ref{Reduction}. But since one can only work
with an approximation of $M_{i,k}$, one need $M_{i,k}$ modulo $p^{N_5}$ with a slightly higher precision $N_5 =
N_4 + 8  \lfloor \mathrm{log}_p N_3  \rfloor +  14$ . We have $N_3 = O(p n)$, $N_4 = O( n)$ and $N_5 = O(n)$.
We work in $\mathbb{Z}_q / p^{N_5} \, $\footnote{More precisely, with $p$-adic precsion $N_4$ but with denominators at most $p^{(N_5-N_4)}$.}, whose elements can be stored in $O(n^2 \mathrm{log} p)$ space and the
arithmetic on it could be done in $\tilde{O} (n^2 \mathrm{log} p)$ bit operations. This gives the algorithm: \\ \vspace{-0.2cm} \\
{\bf Algorithm}
\begin{enumerate}[label=\textbf{Step~\arabic*.},
ref={Step~\arabic*},labelwidth=\dimexpr-\mylen-\labelsep\relax,leftmargin=11.5pt,align=right]
\item \label{step1} Compute $\alpha$ and $\beta$ in Corollary \ref{cor1} modulo $p$.\footnote{In Corollary \ref{cor1}, we only need $\alpha f_y+ \beta f_x \equiv 1 $ modulo $p$ in order to
    compute the lift of Frobenius.}
\item \label{step2} Compute $Z_0$ in Corollary \ref{cor2} using Newton's method, then $F_p(  x^k y^l dx)$ \hspace*{1.375cm} for $l=1,3$ and $ 0 \leq k \leq 2$, all of
    them are modulo $(x^{N_3}, p^{N_4})$.\footnote{In the proof of Lemma \ref{lem1}, we
    showed that the results during the Newton's iteration all have  the same rate of convergence as in Corollary
    \ref{cor2}, so we can work modulo $x^{N_3}$ during the  Newton's iteration.}
\item \label{step3} Use reduction matrices $M_{i,j}$ $( 1 \leq i \leq 3, 2 \leq j \leq N_3 )$ to reduce \hspace*{1.375cm} $F_p(  x^k y^l  dx )$  and get
    $M_{p}$.
\item \label{step4} Compute $M_{q} = M_{p} M_{p}^{\sigma} \cdots M_{p}^{\sigma^{n-1}}$ by repeated squaring.
\item \label{step5} Finally, compute the characteristic polynomial $P_V (X) $ modulo $P^{N_1}$.
\end{enumerate}

\vspace{0.34cm}

\begin{thm}
The above algorithm requires $\tilde{O} ( n^3 p )$ bit operations.
\end{thm}

\begin{proof}
\ref{step1} consists of solving a system of linear
    equations over $\mathbb{F}_q$ of size at most $16$. Hence it requires $\tilde{O} ( n^2 \, \mathrm{log} \, p
    )$ bit operations.
\ref{step2} requires $O(\mathrm{log} \,
    N_4)$ Newton's iterations, and each iteration requires $ \tilde{O} ( N_3 n^2 \mathrm{log} \, p )$ bit operations. Hence this step requires $\tilde{O} ( p n^3  ) $ bit operations.
\ref{step3} requires ${O}( N_3 )$ operations in $\mathbb{Z}_q / p^{N_5}$, hence $\tilde{O}( p  n^3 )$
    bit operations.
\ref{step4} requires $O( \mathrm{log} \, n )$ squarings and the application of the lift of the $p$-th power Frobenius $\sigma : \mathbb{Z}_q \rightarrow \mathbb{Z}_q$ modulo $p^{N_4}$  on matrices of size $6 \times 6$. Squaring requires $\tilde{O} ( n^2 \mathrm{log} \, p)$ bit operations. For $\sigma$, we use Newton's method which needs to
    evaluate a polynomial of degree $n$ with coefficients in $\mathbb{Z}_q / p^{N_4}$, which  requires $O(n)  \tilde{O}(n^2 \, \mathrm{log} \, p)$ bit operations. Hence we need $\tilde{O} (n^3 \, \mathrm{log} \,
    p)$  bit operations in this step.
\ref{step5} requires $\tilde{O}( n^2 \, \mathrm{log} \, p )$ bit operations.
Hence the algorithm requires $\tilde{O}(n^3 p)$ bit operations.
\end{proof}

 If one works directly on
$H^1_{\scriptscriptstyle{MW}}(C_{\mathrm{aff}})$ and denote the precision needed by $N'_i$, then $N_i \approx \frac{2}{3} N'_i$.
Also the matrix $M_p$ is of size $ 6 \times 6$ and $M'_p$ is of size $9 \times 9$.
From these, we give a comparison of speed.
\ref{step2} is reduced by a factor
of $(\frac{2}{3})^2 \approx 0.45$.
In \ref{step3}, we have $6$ differential forms $F_p(  x^k y^l  dx )$, $l=1, \, 3$ and $ 0 \leq k \leq 2$, to reduce. This contributes a factor of $\frac{4}{5}$. (It is $\frac{4}{5}$ instead of $\frac{2}{3}$ because the reductions of $F_p(x^k y^2 dx)$ involve fewer operations than the reductions of $F_p(x^k y^1 dx)$ and $F_p(x^k y^3 dx)$. See the reduction matrices in \ref{Reduction}.) Since each of these $F_p(  x^k y^l  dx )$ is computed modulo $(x^{N_3}, p^{N_4})$ in \ref{step2}, we work with smaller powers on $x$ and fewer $p$-adic precision in \ref{step3}. This means that we have fewer reduction steps and the basic arithmetic operations are faster, which contribute a factor of $(\frac{2}{3})^2$. So \ref{step3} is reduced by a factor of $\frac{4}{5} \cdot (\frac{2}{3})^2 \approx 0.36$. \ref{step4} is reduced by a factor at least of $(\frac{2}{3})^3 \approx 0.3$, due to the smaller size of $M_p$ and fewer precision.

\end{document}